\documentclass[12pt, reqno]{amsart}

\author[M.~Balcerzak]{Marek Balcerzak}
\address{Institute of Mathematics, Lodz University of Technology, ul. W\'{o}lcza\'{n}ska 215, 93-005 Lodz, Poland} % \L{}\'{o}d\'{z}
\email{marek.balcerzak@p.lodz.pl}

\author{Paolo Leonetti}
\address{Department of Decision Sciences, Universit\`a ``Luigi Bocconi'', via Roentgen 1, 20136 Milan, Italy}
\email{leonetti.paolo@gmail.com}
\urladdr{https://sites.google.com/site/leonettipaolo/}


\title{The Baire Category of Subsequences and Permutations which preserve Limit Points}

\usepackage{amsmath}
\usepackage{amssymb}
\usepackage{amsthm}
\usepackage[left=3.5cm, right=3.5cm, paperheight=11.8in]{geometry}
\usepackage{hyperref}
\usepackage{fancyhdr}
\usepackage{enumitem}
\usepackage{comment}
\usepackage{nicefrac}
\usepackage{mathrsfs}
\usepackage{graphicx}
\usepackage[utf8]{inputenc}

\AtBeginDocument{%
   \def\MR#1{}
}

\newtheorem{thm}{Theorem}[section]
\newtheorem{cor}[thm]{Corollary}%[section]
\newtheorem{lem}[thm]{Lemma}
\newtheorem{prop}[thm]{Proposition}

\theoremstyle{definition} 
\newtheorem{defi}[thm]{Definition}%[section]
\let\olddefi\defi
\renewcommand{\defi}{\olddefi\normalfont}
\newtheorem{example}[thm]{Example}
\let\oldexample\example
\renewcommand{\example}{\oldexample\normalfont}
\newtheorem{rmk}[thm]{Remark}
\let\oldrmk\rmk
\renewcommand{\rmk}{\oldrmk\normalfont}

\newcommand{\clusterfin}{\mathrm{L}_x}%{\Gamma_x(\mathrm{Fin})}

\hypersetup{
    pdftitle={The Baire Category of Subsequences and Permutations},
    pdfauthor={Marek Balcerzak and Paolo Leonetti},
    pdfmenubar=false,
    pdffitwindow=true,
    pdfstartview=FitH,
    colorlinks=true,
    linkcolor=blue,
    citecolor=green,
    urlcolor=cyan
}

\providecommand{\MR}[1]{}

\providecommand{\MR}{\relax\ifhmode\unskip\space\fi MR }
\providecommand{\href}[2]{#2}
                                 
\begin{document}

\maketitle
\thispagestyle{empty}

\begin{abstract}
\noindent  Let $\mathcal{I}$ be a meager ideal on $\mathbf{N}$. We show that if $x$ is a sequence with values in a separable metric space then the set of subsequences [resp. permutations] of $x$ which preserve the set of $\mathcal{I}$-cluster points of $x$ is topologically large if and only if every ordinary limit point of $x$ is also an $\mathcal{I}$-cluster point of $x$. The analogue statement fails for all maximal ideals. This extends the main results in [Topology Appl. \textbf{263} (2019), 221--229]. 
As an application, if $x$ is a sequence with values in a first countable compact space which is $\mathcal{I}$-convergent to $\ell$, then the set of subsequences [resp. permutations] which are $\mathcal{I}$-convergent to $\ell$ is topologically large if and only if $x$ is convergent to $\ell$ in the ordinary sense. Analogous results hold for $\mathcal{I}$-limit points, provided $\mathcal{I}$ is an analytic P-ideal.
\end{abstract}

\section{Introduction}\label{sec:intro}

A classical result of Buck \cite{MR9997} states that, if $x$ is real sequence, then \textquotedblleft almost every\textquotedblright\, subsequence of $x$ has the same set of ordinary limit points of the original sequence $x$, in a measure sense. The aim of this note is to prove its topological analogue and non-analogue in the context of ideal convergence. 
%, thus extending the results in \cite{MR3968131, MR3950736}.

Let $\mathcal{I}$ be an \emph{ideal} on the positive integers $\mathbf{N}$, that is, a family a subsets of $\mathbf{N}$ closed under subsets and finite unions. Unless otherwise stated, it is also assumed that $\mathcal{I}$ contains the ideal $\mathrm{Fin}$ of finite sets and it is different from the power set $\mathcal{P}(\mathbf{N})$. 
$\mathcal{I}$ is a P-ideal if it is $\sigma$-directed modulo finite sets, i.e., for every sequence $(A_n)$ of sets in $\mathcal{I}$ there exists $A \in \mathcal{I}$ such that $A_n\setminus A$ is finite for all $n$. 
%Let $\mathcal{I}^\star:=\{A\subseteq \mathbf{N}: A^c \in \mathcal{I}\}$ be its dual filter. 
%and $\mathcal{I}^+:=\{A\subseteq \mathbf{N}: A\notin \mathcal{I}\}$ be the family of $\mathcal{I}$- positive sets. 
We regard ideals as subsets of the Cantor space $\{0,1\}^{\mathbf{N}}$, hence we may speak about their topological complexities. In particular, an ideal can be $F_\sigma$, 
%$F_{\sigma\delta}$, 
analytic, etc. 
%Given ideals $\mathcal{I}, \mathcal{J}$ on $\mathbf{N}$, we let their Fubini product be the ideal on $\mathbf{N}^2$ defined by 
%$$
%\mathcal{I}\times \mathcal{J}:=\{A\subseteq \mathbf{N}^2: \{n \in \mathbf{N}: \{k \in \mathbf{N}: (n,k) \in A\}\notin \mathcal{J}\} \in \mathcal{I}\}.
%$$ 
Among the most important ideals, we find the family of asymptotic density zero sets 
$$
\textstyle \mathcal{Z}:=\{A\subseteq \mathbf{N}: \lim_{n\to \infty}\frac{1}{n}|A\cap [1,n]|=0\}.
$$
We refer to 
%\cite{MR2777744, MR3920747} 
%for recent surveys on ideals and associated filters.
\cite{MR2777744} 
for a recent survey on ideals and associated filters.

Let $x=(x_n)$ be a sequence taking values in a topological space $X$, which will be always assumed to be Hausdorff. Then $\ell \in X$ is an $\mathcal{I}$\emph{-cluster point} of $x$ if 
$$
\{n \in \mathbf{N}: x_n \in U\} \notin \mathcal{I}
$$
for each neighborhood $U$ of $\ell$. The set of $\mathcal{I}$-cluster points of $x$ is denoted by $\Gamma_x(\mathcal{I})$. 
Moreover, $\ell \in X$ is an $\mathcal{I}$\emph{-limit point} of $x$ if there exists a subsequence $(x_{n_k})$ such that
$$
\lim_{k\to \infty}x_{n_k}=\ell
\,\,\,\text{ and }\,\,\,
\{n_k: k \in \mathbf{N}\} \notin \mathcal{I}.
$$
The set of $\mathcal{I}$-limit points is denoted by $\Lambda_x(\mathcal{I})$. Statistical cluster points and statistical limits points (that is, $\mathcal{Z}$-cluster points and $\mathcal{Z}$-limit points) of real sequences were introduced by Fridy in \cite{MR1181163} and studied by many authors, see e.g. \cite{MR1372186, 
MR2463821, 
MR1416085, MR1838788, MR1260176, MR1924673}. 
It is worth noting that ideal cluster points have been studied much before under a different name. Indeed, as it follows by \cite[Theorem 4.2]{MR3920799}, they correspond to classical ``cluster points'' of a filter $\mathscr{F}$ (depending on $x$) on the underlying space, cf. \cite[Definition 2, p.69]{MR1726779}. 
Let also $\mathrm{L}_x:=\Gamma_x(\mathrm{Fin})$ be the set of accumulation points of $x$, and note that $\mathrm{L}_x=\Lambda_x(\mathrm{Fin})$ if $X$ is first countable. 
%let $\mathrm{L}_x:=\Gamma_x(\mathrm{Fin})$ denote the set of ordinary limit points of $x$. 
%See \cite{MR3883171, MR3920799} for characterizations of $\mathcal{I}$-cluster points and their relationship with $\mathcal{I}$-limit points. 
Hence $\Lambda_x(\mathcal{I})\subseteq \Gamma_x(\mathcal{I}) \subseteq \mathrm{L}_x$. 
We refer the reader to \cite{MR3920799} for characterizations of $\mathcal{I}$-cluster points and  \cite{MR3883171} for their relation with $\mathcal{I}$-limit points. 
Lastly, we recall that the sequence $x$ is said to be $\mathcal{I}$\emph{-convergent} to $\ell \in X$, shortened as $x \to_{\mathcal{I}} \ell$, if 
$$
\{n \in \mathbf{N}: x_n \notin U\} \in \mathcal{I}
$$
for each neighborhood $U$ of $\ell$. Assuming that $X$ is first countable, it follows by \cite[Corollary 3.2]{MR3920799} that, if $x \to_{\mathcal{I}} \ell$ then $\Lambda_x(\mathcal{I})=\Gamma_x(\mathcal{I})=\{\ell\}$, provided that $\mathcal{I}$ is a P-ideal. 
In addition, if $X$ is also compact, then $x \to_{\mathcal{I}} \ell$ is in fact equivalent to $\Gamma_x(\mathcal{I})=\{\ell\}$, even if $\mathcal{I}$ is not a P-ideal, cf. \cite[Corollary 3.4]{MR3920799}.

Let $\Sigma$ be the sets of strictly increasing functions on $\mathbf{N}$, that is, 
%let $\Sigma$ and $\Pi$ be the sets of strictly increasing functions
$$
\Sigma:=\{\sigma \in \mathbf{N}^{\mathbf{N}}: \forall n \in \mathbf{N}, \sigma(n)<\sigma(n+1)\};
%\text{ and }
%\Pi:=\{\pi \in  \mathbf{N}^{\mathbf{N}}: \pi \text{ is a bijection}\}.
$$
also, let $\Pi$ be the sets of permutations of $\mathbf{N}$, that is, 
$$
\Pi:=\{\pi \in  \mathbf{N}^{\mathbf{N}}: \pi \text{ is a bijection}\}.
$$
% let $\Sigma \subseteq \mathbf{N}^{\mathbf{N}}$ be the set of strictly increasing functions $\sigma: \mathbf{N}\to \mathbf{N}$, and $\Pi \subseteq \mathbf{N}^{\mathbf{N}}$ be the set of permutations $\pi$ of $\mathbf{N}$. 
Note that both $\Sigma$ and $\Pi$ are $G_\delta$-subsets of the Polish space $\mathbf{N}^{\mathbf{N}}$, hence they are Polish spaces as well by Alexandrov's theorem; in particular, they are not meager in themselves, cf. \cite[Chapter 2]{MR1619545}. 
Given a sequence $x$ and $\sigma \in \Sigma$, we denote by $\sigma(x)$ the subsequence $(x_{\sigma(n)})$. Similarly, given $\pi \in \Pi$, we write $\pi(x)$ for the rearranged sequence $(x_{\pi(n)})$. 
We identify each subsequence of $(x_{k_n})$ of $x$ with the function $\sigma \in \Sigma$ defined by $\sigma(n)=k_n$ for all $n \in \mathbf{N}$ and, similarly, each rearranged sequence $(x_{\pi(n)})$ with the permutation $\pi \in \Pi$, cf. \cite{MR3568092, MR3739371, MR1924673}.
%This gives clearly a bijection between $\Sigma$ [resp. $\Pi$] and the set of subsequences of $x$ [resp. permutations of $x$], cf. \cite{MR3568092, MR3739371, MR1924673}. 

We will show that 
if $\mathcal{I}$ is a meager ideal and $x$ is a sequence with values in a separable metric space then the set of subsequences (and permutations) of $x$ which preserve the set of $\mathcal{I}$-cluster points of $x$ is not meager if and only if every ordinary limit point of $x$ is also an $\mathcal{I}$-cluster point of $x$ (Theorem \ref{thm:theoremcluster}). 
A similar result holds for $\mathcal{I}$-limit points, provided that $\mathcal{I}$ is an analytic P-ideal (Theorem \ref{thm:theoremlimit}). 
Putting all together, this strenghtens all the results contained in \cite{MR3968131} and answers an open question therein.  
As a byproduct, we obtain a characterization of meager ideals (Proposition \ref{lem:talagrand}). 
Lastly, the analogue statements fails for all maximal ideals (Example \ref{ex:maximalcounterexample}). 

%%%%%%%%%%%%%%%%%%%%%%%%%%%%%%%%%%%%%%%%
\section{Main results}\label{sec:mainresults}

\subsection{$\mathcal{I}$-cluster points.} 
It has been shown in \cite{MR3968131} that, from a topological viewpoint, almost all subsequences of $x$ preserve the set of $\mathcal{I}$-cluster points, provided that $\mathcal{I}$ is \textquotedblleft well separated\textquotedblright\, from its dual filter $\mathcal{I}^\star:=\{A\subseteq \mathbf{N}: A^c \in \mathcal{I}\}$; that is, 
$$%\begin{equation}\label{eq:setclaimed}
\Sigma_x(\mathcal{I}):=\left\{\sigma \in \Sigma: \Gamma_{\sigma(x)}(\mathcal{I})=\Gamma_x(\mathcal{I})\right\}
$$ %\end{equation}
is comeager, cf. also  \cite{MR3950736} for the case $\mathcal{I}=\mathcal{Z}$ and \cite{MR3879311} for a measure theoretic analogue. We will extend this result to all meager ideals. In addition, we will see that the same holds also for
$$
\Pi_x(\mathcal{I}):=\left\{\pi \in \Pi: \Gamma_{\pi(x)}(\mathcal{I})=\Gamma_x(\mathcal{I})\right\}.
$$

Here, given $\mathcal{A}, \mathcal{B}, \mathcal{C} \subseteq \mathcal{P}(\mathbf{N})$, we say that $\mathcal{A}$ is separated from $\mathcal{C}$ by $\mathcal{B}$ if $\mathcal{A}\subseteq \mathcal{B}$ and $\mathcal{B} \cap \mathcal{C}=\emptyset$. In particular, an ideal $\mathcal{I}$ is $F_\sigma$-separated from its dual filter $\mathcal{I}^\star$ if there exists an $F_\sigma$-set $\mathcal{B} \subseteq \mathcal{P}(\mathbf{N})$ such that $\mathcal{I}\subseteq \mathcal{B}$ and $\mathcal{B} \cap \mathcal{I}^\star=\emptyset$ (with the language of \cite{MR2520152, MR2990109}, the filter $\mathcal{I}^\star$ has rank $\le 1$). 
\begin{thm}\label{thm:oldtheoremcluster}
 \cite[Theorem 2.1]{MR3968131} Let $x$ be a sequence in a first countable space $X$ such that all closed sets are separable and let $\mathcal{I}$ be an ideal which is $F_\sigma$-separated from its dual filter $\mathcal{I}^\star$. Then 
$\Sigma_x(\mathcal{I})$ 
%$$%\begin{equation}\label{eq:setclaimed}
%\Sigma_x(\mathcal{I}):=\left\{\sigma \in \Sigma: \Gamma_{\sigma(x)}(\mathcal{I})=\Gamma_x(\mathcal{I})\right\}
%$$ %\end{equation}
is not meager if and only if $\Gamma_x(\mathcal{I})=\clusterfin$. Moreover, in this case, it is comeager.
\end{thm}
As it has been shown in \cite[Corollary 1.5]{MR1758325}, the family of ideals $\mathcal{I}$ which are $F_{\sigma}$-separated from $\mathcal{I}^\star$ includes all $F_{\sigma\delta}$-ideals. 
In addition, a Borel ideal 
%$\mathcal{I}$ 
is $F_\sigma$-separated from its dual filter 
%$\mathcal{I}^\star$ 
if and only if it does not contain an isomorphic copy of $\mathrm{Fin}\times \mathrm{Fin}$ (which can be represented as an ideal on $\mathbf{N}$ as 
$$
\{A\subseteq \mathbf{N}: 
%\exists n_0 \in \mathbf{N}, \forall n \ge n_0, 
\forall^\infty n \in \mathbf{N}, \{a \in A: \nu_2(a)=n\}\in \mathrm{Fin} \}
$$
where $\nu_2(n)$ stands for the $2$-adic valuation of $n$), see  \cite[Theorem 4]{MR2491780}. 
In particular, $\mathrm{Fin}\times \mathrm{Fin}$ is a $F_{\sigma\delta\sigma}$-ideal which is not $F_\sigma$-separated from its dual filter. 
For related results on $F_\sigma$-separation, see \cite[Proposition 3.6]{MR3090420} and \cite{MR2520149}.

We show that the analogue of Theorem \ref{thm:oldtheoremcluster} holds for all meager ideals. Hence, this includes new cases as, for instance, $\mathcal{I}=\mathrm{Fin}\times \mathrm{Fin}$. 

It is worth noting that every meager ideal $\mathcal{I}$ is $F_\sigma$-separated from the Fr\'{e}chet filter $\mathrm{Fin}^\star$ (see Proposition \ref{lem:talagrand} below), hence $\mathcal{I}$ is $F_\sigma$-separated from $\mathcal{I}^\star$. This implies that our result is a proper generalization of Theorem \ref{thm:oldtheoremcluster}:
%Lastly, we see that the analoguous result holds also for:
%$$
%\Pi_x(\mathcal{I}):=\left\{\pi \in \Pi: \Gamma_{\pi(x)}(\mathcal{I})=\Gamma_x(\mathcal{I})\right\}.
%$$

\begin{thm}\label{thm:theoremcluster}
Let $x$ be a sequence in a first countable space $X$ such that all closed sets are separable and let $\mathcal{I}$ be a meager ideal. Then 
the following are equivalent: 
\begin{enumerate}[label={\rm (\textsc{c}\arabic{*})}]
\item \label{item:s1} $\Sigma_x(\mathcal{I})$ is comeager in $\Sigma$\textup{;}
\item \label{item:s2} $\Sigma_x(\mathcal{I})$ is not meager in $\Sigma$\textup{;}
\item \label{item:p1} $\Pi_x(\mathcal{I})$ is comeager in $\Pi$\textup{;}
\item \label{item:p2} $\Pi_x(\mathcal{I})$ is not meager in $\Pi$\textup{;}
%\item \label{item:s1} $\left\{\sigma \in \Sigma: \Gamma_{\sigma(x)}(\mathcal{I})=\Gamma_x(\mathcal{I})\right\}$ is comeager\textup{;}
%\item \label{item:s2} $\left\{\sigma \in \Sigma: \Gamma_{\sigma(x)}(\mathcal{I})=\Gamma_x(\mathcal{I})\right\}$ is not meager\textup{;}
%\item $\left\{\pi \in \Pi: \Gamma_{\pi(x)}(\mathcal{I})=\Gamma_x(\mathcal{I})\right\}$ is comeager\textup{;}
%\item $\left\{\pi \in \Pi: \Gamma_{\pi(x)}(\mathcal{I})=\Gamma_x(\mathcal{I})\right\}$ is not meager\textup{;}
\item \label{item:s3} $\Gamma_x(\mathcal{I})=\clusterfin$\textup{.}
\end{enumerate}
%the set \eqref{eq:setclaimed} 
%%$$
%%\left\{\omega \in (0,1]: \Gamma_{x \upharpoonright \omega}(\mathcal{I})=\Gamma_x(\mathcal{I})\right\}
%%$$
%is not meager if and only if $\Gamma_x(\mathcal{I})=\clusterfin$. 
%Moreover, in this case, it is comeager.
\end{thm}

It is worth to spend some comments on the class $\mathscr{C}$ of first countable spaces such that all closed sets are separable. It is clear that separable metric spaces belong to $X$. However, $\mathscr{C}$ contains also nonmetrizable spaces:

\begin{example}\label{ex:nonmetrizable}
It has been shown by Ostaszewski in \cite{MR438292} that there exists a topological space $X$ which is hereditarily separable (i.e., all subsets of $X$ are separable), first countable, countably compact, and not compact, cf. also \cite{MR428245}. In particular, $X \in \mathscr{C}$. However, $X$ is not second countable (indeed, in the opposite, the notions of compactness and countably compactness would coincide). Considering that all separable metric spaces are second countable, it follows that $X$ is not metrizable.
\end{example}

Moreover, there exists a separable first countable space outside $\mathscr{C}$:
\begin{example}\label{ex:separablenonhereditary}
Let $X$ be the Sorgenfrey plane, that is, the product of two copies of the real line $\mathbf{R}$ endowed with the lower limit topology. 
% (i.e., the product of two Sorgenfrey lines). 
It is known that $X$ is first countable and separable. 
Moreover, the anti-diagonal $\{(x,-x): x \in \mathbf{R}\}$ is a closed uncountable discrete subspace of $X$, hence not separable, cf. \cite[pp. 103--105]{MR1382863}. Therefore $X\notin \mathscr{C}$. 
A similar example is the Moore plane with the tangent disk topology, cf. \cite[p. 176]{MR1382863}.
\end{example}

We can also show that there exists a first countable space outside $\mathscr{C}$ which satisfies the statement of Theorem \ref{thm:theoremcluster}:
\begin{example}\label{example:powerset}
Let $X$ be an uncountable set, endowed the discrete topology. Then $X$ is nonseparable first countable space, so that $X\notin \mathscr{C}$. Thanks to Remark \ref{rmk:uncountable} below, Theorem \ref{thm:theoremcluster} holds if the separability of all closed sets is replaced by the condition that $\mathrm{L}_x$ is countable \emph{for each} sequence $x$ taking values in $X$. Indeed, the latter is verified because $\mathrm{L}_x \subseteq \{x_n: n \in \mathbf{N}\}$.
\end{example}

We leave as open question 
%for the interested reader 
whether there exists a topological space $X$ (necessarily outside $\mathscr{C}$) and a meager ideal $\mathcal{I}$ for which Theorem \ref{thm:theoremcluster} fails. 
It is well possible that our main result extends beyond $\mathscr{C}$, as it happened very recently with related results, see e.g. the improvement of \cite[Theorem 4.2]{MR3955010} in \cite[Lemma 2.2]{MR4126774}.

%\medskip
%
%\textcolor{blue}{\underline{(In particular, open question:)} Fix $\mathcal{I}=\mathrm{Fin}$ so that condition \ref{item:s3} is verified by definition. Is it true that there exists a sequence $x$ in a topological space such that $\{\sigma \in \Sigma: \mathrm{L}_{\sigma(x)}=\mathrm{L}_x\}$ is not comeager in $\Sigma$?}
%
%\bigskip

%Note that the standing hypotheses hold if $X$ is a separable metric space. 
%%
%At this point, 
Lastly, 
one may ask whether Theorem \ref{thm:theoremcluster} holds for all ideals.  
%one may ask whether the same statement holds for all ideals. 
% $\mathcal{I}$. 
We show in the following example that the answer is negative: % whenever $\mathcal{I}$ is maximal.
\begin{example}\label{ex:maximalcounterexample}
Let $\mathcal{I}$ be a maximal ideal. Hence there exists a unique $A \in \{2\mathbf{N}+1,2\mathbf{N}+2\}$ such that $A \in \mathcal{I}$. Set $X=\mathbf{R}$. Let $x$ be the characteristic function of $A$, i.e., $x_n=1$ if $n \in A$ and $x_n=0$ otherwise. Then $x \to_{\mathcal{I}} 0$. In particular, $\Gamma_x(\mathcal{I})=\{0\}$. Note that a subsequence $\sigma(x)$ is $\mathcal{I}$-convergent to $0$ 
%if and only if $\sigma^{-1}[A] \in \mathcal{I}$ 
if and only if $\Gamma_{\sigma(x)}=\{0\}$. Then 
$$
\Sigma_x(\mathcal{I})=\{\sigma \in \Sigma: \sigma^{-1}[A] \in \mathcal{I}\}.
$$
Considering that $\sigma^{-1}[A] \cup \sigma^{-1}[A-1]$ is cofinite, we have either $\sigma^{-1}[A] \in \mathcal{I}$ or $\sigma^{-1}[A-1]\in \mathcal{I}$.  Let $T: \Sigma \to \Sigma$ be the embedding defined by $\sigma \mapsto \sigma+1$, so that $\Sigma$ is homeomorphic to the open set $T[\Sigma]=\{\sigma \in \Sigma: \sigma(1)\ge 2\}$. Notice that
\begin{displaymath}
\begin{split}
T[\Sigma_x(\mathcal{I})]&=\{T(\sigma): \sigma \in \Sigma_x(\mathcal{I})\}
=\{\sigma+1 \in \Sigma: \sigma^{-1}[A] \in \mathcal{I}\}\\
&=\{\sigma \in \Sigma: \sigma^{-1}[A-1] \in \mathcal{I}\} \cap T[\Sigma],
\end{split}
\end{displaymath}
which implies that the open set $T[\Sigma]$ is contained in $\Sigma_x(\mathcal{I}) \cup T[\Sigma_x(\mathcal{I})]$. Therefore both $\Sigma_x(\mathcal{I})$ and $T[\Sigma_x(\mathcal{I})]$ are not meager.

A similar example can be found for $\Pi_x(\mathcal{I})$, replacing the embedding $T$ with the homeomorphism $H: \Pi\to \Pi$ defined by $H(\pi)(2n)=2n-1$ and $H(\pi)(2n-1)=2n$ for all $n \in \mathbf{N}$.
\end{example}

As noted by the referee, the equivalences \ref{item:s1} $\Longleftrightarrow$ \ref{item:s2} and \ref{item:p1} $\Longleftrightarrow$ \ref{item:p2}, and their analogues in the next results, can be viewed a consequence a general topological $0\text{-}1$ law stating that every tail set with the property of Baire is either meager or comeager, see \cite[Theorem 21.4]{MR584443}. However, it seems rather difficult to show that the tail sets $\Sigma_x(\mathcal{I})$ and $\Pi_x(\mathcal{I})$ have the property of Baire, provided that $\mathcal{I}$ is meager (note that this is surely false if $\mathcal{I}$ is a maximal ideal, as it follows by Example \ref{ex:maximalcounterexample}).
 
%\textcolor{blue}{[Ask Marek]}

As an application of our results, 
%we obtain that, 
if $x$ is $\mathcal{I}$-convergent to $\ell$, then the set of subsequences [resp., rearrangements] of $x$ which are $\mathcal{I}$-convergent to $\ell$ is not meager if and only if $x$ is convergent (in the classical sense) to $\ell$. 
%(Here, a sequence $x$ is said to be $\mathcal{I}$-convergent to $\ell$, shortened as $x \to_{\mathcal{I}} \ell$, if $\{n \in \mathbf{N}: x_n \notin U\} \in \mathcal{I}$ for each neighborhood $U$ of $\ell$.) 
This is 
%\textcolor{red}{a generalization of \cite[Theorem 1]{Miller2019} and it is}, 
somehow related to \cite[Theorem 2.1]{MR3568092} and \cite[Theorem 1.1]{MR3739371}; cf. also \cite[Theorem 3]{MR1260176} for a measure theoretical non-analogue.
\begin{cor}\label{thm:applicationidealconvergence}
Let $x$ be a sequence in a first countable compact space $X$. Let $\mathcal{I}$ be a meager ideal and assume that $x$ is $\mathcal{I}$-convergent to $\ell \in X$. Then the following are equivalent:
\begin{enumerate}[label={\rm (\textsc{i}\arabic{*})}]
\item \label{item:c2} $\{\sigma \in \Sigma: \sigma(x) \to_{\mathcal{I}} \ell\}$ is comeager in $\Sigma$\textup{;}
\item \label{item:c3} $\{\sigma \in \Sigma: \sigma(x) \to_{\mathcal{I}} \ell\}$ is not meager in $\Pi$\textup{;}
\item \label{item:c4} $\{\pi \in \Pi: \pi(x) \to_{\mathcal{I}} \ell\}$ is comeager in $\Sigma$\textup{;}
\item \label{item:c5} $\{\pi \in \Pi: \pi(x) \to_{\mathcal{I}} \ell\}$ is not meager in $\Pi$\textup{;}
\item \label{item:c1} $\lim_n x_n=\ell$\textup{.}
\end{enumerate}
\end{cor}

%\textcolor{red}{A remark is in order. By Corollary \ref{thm:applicationidealconvergence}, if a sequence $x$ is $\mathcal{I}$-convergent to $\ell$, then it is \emph{not} true, in general, that the set of subsequences which are $\mathcal{I}$-convergent to $\ell$ is comeager. However, the converse holds. Indeed, assume that $\mathcal{S}:=\{\sigma \in \Sigma: \sigma(x)\to_{\mathcal{I}}\ell\}$ is comeager. Let $T:\Sigma\to \Sigma$ be the homeomorphism which associates to each $\sigma \in \Sigma$ the unique $\sigma^\prime \in \Sigma$ such that $\sigma^\prime[\mathbf{N}]=\sigma[\mathbf{N}]^c$. Then also $T[\mathcal{S}]$ is comeager, which implies that $\mathcal{S} \cap T[\mathcal{S}] \neq \emptyset$. Fix $\sigma^\star \in \mathcal{S} \cap T[\mathcal{S}]$, so that $\sigma^\star(x)\to_{\mathcal{I}}\ell$ and $T(\sigma^\star)(x)\to_{\mathcal{I}}\ell$. Therefore $x\to_{\mathcal{I}}\ell$.

The proofs of Theorem \ref{thm:theoremcluster} and Corollary \ref{thm:applicationidealconvergence} follow in Section \ref{sec:proofcluster}.

\subsection{$\mathcal{I}$-limit points.} 
Given a sequence $x$ and an ideal $\mathcal{I}$, define %the analogue sets
$$
\tilde{\Sigma}_x(\mathcal{I}):=\{\sigma \in \Sigma: \Lambda_{\sigma(x)}(\mathcal{I})=\Lambda_x(\mathcal{I})\}
%\text{ and }
%\tilde{\Pi}_x(\mathcal{I}):=\{\pi \in \Pi: \Lambda_{\pi(x)}(\mathcal{I})=\Lambda_x(\mathcal{I})\}.
$$
and its analogue for permutations
$$
\tilde{\Pi}_x(\mathcal{I}):=\{\pi \in \Pi: \Lambda_{\pi(x)}(\mathcal{I})=\Lambda_x(\mathcal{I})\}.
$$

It has been shown in \cite{MR3968131} that, in the case of $\mathcal{I}$-limit points, the counterpart of Theorem \ref{thm:oldtheoremcluster} holds for generalized density ideals. Here, an ideal $\mathcal{I}$ is said to be a \emph{generalized density ideal} if there exists a sequence $(\mu_n)$ of submeasures with finite and pairwise disjoint supports such that $\mathcal{I}=\{A\subseteq \mathbf{N}: \lim_n \mu_n(A)=0\}$. More precisely:
\begin{thm}\label{thm:oldtheoremlimit}
 \cite[Theorem 2.3]{MR3968131} Let $x$ be a sequence in a first countable space $X$ such that all closed sets are separable and let $\mathcal{I}$ be generalized density ideal. Then $\tilde{\Sigma}_x(\mathcal{I})$ is not meager if and only if $\Lambda_x(\mathcal{I})=\clusterfin$. Moreover, in this case, it is comeager.
\end{thm}

See \cite{Leo17b} for a measure theoretic analogue. It has been left as open question to check, in particular, whether the same statement holds for analytic P-ideals. We show that the answer is affirmative. 

Note that this is strict generalization, as every generalized density ideal is an analytic P-ideal and there exists an analytic P-ideal which is not a generalized density ideal, see e.g. \cite{MR3436368}. In addition, the same result holds for permutations.
\begin{thm}\label{thm:theoremlimit}
Let $x$ be a sequence in a first countable space $X$ such that all closed sets are separable and let $\mathcal{I}$ be an analytic P-ideal. Then 
the following are equivalent: 
\begin{enumerate}[label={\rm (\textsc{L}\arabic{*})}]
\item \label{item:l1} $\tilde{\Sigma}_x(\mathcal{I})$ is comeager in $\Sigma$\textup{;}
\item \label{item:l2} $\tilde{\Sigma}_x(\mathcal{I})$ is not meager in $\Sigma$\textup{;}
\item \label{item:llll1} $\tilde{\Pi}_x(\mathcal{I})$ is comeager in $\Pi$\textup{;}
\item \label{item:llll2} $\tilde{\Pi}_x(\mathcal{I})$ is not meager in $\Pi$\textup{;}
\item \label{item:l3} $\Gamma_x(\mathcal{I})=\clusterfin$\textup{.}
\end{enumerate}
\end{thm}
Note that the same Example \ref{ex:maximalcounterexample} shows that the analogue of Theorem \ref{thm:theoremlimit} fails for all maximal ideals. 
The proof of Theorem \ref{thm:theoremlimit} follows in Section \ref{sec:prooflimit}.

We leave as an open question to check whether Theorem \ref{thm:theoremlimit} may be extended to all meager ideals.

%\textcolor{cyan}{Corollaries for $\mathcal{I}^\star$-convergence?}

%\textcolor{red}{Q. Analogues for $\mathcal{I}$-limit points? In particular, for Lemma \ref{lem:fsigmanew} below. At least for $\mathcal{I}$ analytic P-ideal so that we know that $\Lambda_x(\mathcal{I})$ is an $F_\sigma$-set.}

%\begin{thm}\label{thm:permutations}
%Let $X$ be a first countable space such that all closed sets are separable and let $\mathcal{I}$ be a meager ideal. Then 
%$$
%\left\{\pi \in \Pi: \Gamma_{\pi(x)}(\mathcal{I})=\Gamma_x(\mathcal{I})\right\}
%$$
%is not meager if and only if $\Gamma_x(\mathcal{I})=\clusterfin$. Moreover, in this case, it is comeager.
%\end{thm}

%\clearpage
\section{Proofs for $\mathcal{I}$-cluster points}\label{sec:proofcluster}

%We recall the following well-known characterization of meager ideals due to Talagrand. 
We start with a characterization of meager ideals 
%, which may be of independent interest 
(to the best of our knowledge, condition \ref{item:m3} is novel). Here, 
a set $\mathcal{A}\subseteq \mathcal{P}(\mathbf{N})$ 
%a subset of $\mathcal{P}(\mathbf{N})$ 
is called \emph{hereditary} if it is closed under subsets.
\begin{prop}\label{lem:talagrand}
Let $\mathcal{I}$ be an ideal on $\mathbf{N}$. Then the following are equivalent:
\begin{enumerate}[label={\rm (\textsc{m}\arabic{*})}]
\item \label{item:m1} $\mathcal{I}$ is meager\textup{;}
\item \label{item:m2} There exists a strictly increasing sequence $(\iota_n)$ of positive integers such that $A \notin \mathcal{I}$ whenever $\mathbf{N} \cap [\iota_n,\iota_{n+1}) \subseteq A$ for infinitely many $n \in \mathbf{N}$\textup{;}
\item \label{item:m3} $\mathcal{I}$ is $F_\sigma$-separated from the Fr\'{e}chet filter $\mathrm{Fin}^\star$\textup{.}
%\item \label{item:m4} $\mathcal{I}$ is separated from the Fr\'{e}chet filter $\mathrm{Fin}^\star$ by $\bigcup_k F_k$, where each $F_k$ is a hereditary closed set invariant under finite changes\textup{.}
\end{enumerate}
%An ideal $\mathcal{I}$ is meager if and only if there exists a stricly increasing sequence $(\iota_n)$ of positive integers such that $A \in \mathcal{I}^+$ whenever $\mathbf{N} \cap [\iota_n,\iota_{n+1}) \subseteq A$ for infinitely many $n \in \mathbf{N}$. 
\end{prop}
\begin{proof}
\ref{item:m1} $\Longleftrightarrow$ \ref{item:m2} See \cite[Theorem 2.1]{MR579439}; cf. also \cite[Theorem 4.1.2]{MR1350295}.

\ref{item:m2} $\implies$ \ref{item:m3} Define $I_n:=\mathbf{N} \cap [\iota_n,\iota_{n+1})$ for all $n \in \mathbf{N}$. Then $\mathcal{I}\subseteq F$, where $F:=\bigcup_k F_k$ and
\begin{equation}\label{eq:definitionFk}
\forall k \in \mathbf{N},\quad F_k:=\bigcap_{n\ge k}\{A\subseteq \mathbf{N}: I_n \not\subseteq A\}.
\end{equation}
Note that each $F_k$ is closed and it does not contain any cofinite set. Therefore $\mathcal{I}$ is separated from $\mathrm{Fin}^\star$ by the $F_\sigma$-set $F$.

\begin{comment}
\ref{item:m3} $\Longleftrightarrow$ \ref{item:m4} Suppose that $\mathcal{I}$ is separated from $\mathrm{Fin}^\star$ by $\bigcup_k C_k$, where each $C_k$ is closed. Then it is enough to set 
%$F_k:=\{A\subseteq \mathbf{N}: A\setminus C_k \in \mathrm{Fin}\}$ for all $k \in \mathbf{N}$. 
$$
\forall k \in \mathbf{N},\quad F_k:=\{A\subseteq \mathbf{N}: A\setminus C_k \in \mathrm{Fin}\}.
$$
In particular, $F_k$ does not contain any cofinite set. The converse is obvious. %for each $k \in \mathbf{N}$. 
\end{comment}

%\ref{item:m4} $\implies$ \ref{item:m3} It is obvious.

\ref{item:m3} $\implies$ \ref{item:m1} Suppose that there exists a sequence $(F_k)$ of closed sets in $\{0,1\}^{\mathbf{N}}$ such that $\mathcal{I}\subseteq F:=\bigcup_k F_k$ and $F \cap \mathrm{Fin}^\star=\emptyset$. Then each $F_k$ has empty interior (otherwise it would contain a cofinite set). We conclude that $\mathcal{I}$ is contained in a countable union of nowhere dense sets.
\end{proof}

\begin{comment}
The above characterization is reminescent of an open question of Mazur \cite{MR1124539}, cf. also \cite[p. 220]{MR1758325}: Is every $F_{\sigma\delta}$-ideal contained in a hereditary $F_{\sigma}$-set $F$ such that $X\cup Y$ is not cofinite for all $X,Y \in F$?
\end{comment}

%In addition, it 
It 
is clear that condition \ref{item:m3} is weaker than the extendability of $\mathcal{I}$ to a $F_{\sigma}$-ideal. 
%\begin{enumerate}[label={\rm (\textsc{m}\arabic{*}')}]
%\setcounter{enumi}{2}
%\item \label{item:m3prime} $\mathcal{I}$ can be extended to a $F_{\sigma}$-ideal.
%\end{enumerate} 
For characterizations and related results 
of the latter property, 
see e.g. \cite[Theorem 4.4]{MR3142391} and \cite[Theorem 3.3]{MR2557668}. 
%\cite{MR3640040}. 

\begin{lem}\label{lem:0referee}
Let $\mathcal{I}$ be a meager ideal. Then 
$$
\{\sigma \in \Sigma: \sigma^{-1}[A]\notin \mathcal{I}\}
\quad \text{ and }\quad 
\{\pi \in \Pi: \pi^{-1}[A]\notin \mathcal{I}\}
$$
are comeager for each infinite set $A\subseteq \mathbf{N}$.
\end{lem}
\begin{proof}
Fix an infinite set $A\subseteq \mathbf{N}$. As in the proof of Proposition \ref{lem:talagrand}, we can define intervals $I_n:=\mathbf{N}\cap [\iota_n, \iota_{n+1})$ for all $n \in \mathbf{N}$ such that a set $S\subseteq \mathbf{N}$ does not belong to $\mathcal{I}$ whenever $S$ contains infinitely many intervals $I_n$s. At this point, for each $n,k \in \mathbf{N}$, define the sets
$$
\textstyle 
X_k:=\{\sigma \in \Sigma: I_k\subseteq \sigma^{-1}[A]\}
\quad \text{ and }\quad 
Y_n:=\bigcup_{k\ge n}X_k.
$$
Note that each $X_k$ is open and that each $Y_n$ is open and dense. Therefore the set $\bigcap_n Y_n$, which can be rewritten as $\{\sigma \in \Sigma: \sigma^{-1}[A]\notin \mathcal{I}\}$, is comeager. 
The proof that $\{\pi \in \Pi: \pi^{-1}[A]\notin \mathcal{I}\}$ is comeager is analogous.
\end{proof}

\begin{lem}\label{lem:fsigmanew}
Let $x$ be a sequence in a first countable space $X$ and let $\mathcal{I}$ be a meager ideal. 
Then 
$$
S(\ell):=\left\{\sigma \in \Sigma: \ell \in \Gamma_{\sigma(x)}(\mathcal{I})\right\} 
\quad \text{ and }\quad 
P(\ell):=\left\{\pi \in \Pi: \ell \in \Gamma_{\pi(x)}(\mathcal{I})\right\}
$$ 
are comeager for each $\ell \in \clusterfin$.
\end{lem}
\begin{proof}
Assume that $\clusterfin \neq \emptyset$, otherwise there is nothing to prove. 
Fix $\ell \in \clusterfin$ and let $(U_m)$ be a decreasing local base at $\ell$.  For each $m \in \mathbf{N}$, define the infinite set $A_m:=\{n \in \mathbf{N}:  x_n \in U_m\}$. Thanks to Lemma \ref{lem:0referee}, the set $B_m:=\{\sigma \in \Sigma: \sigma^{-1}[A_m]\notin \mathcal{I}\}$ is comeager. Since $S(\ell)$ can be rewritten as $\bigcap_m B_m$, it follows that $S(\ell)$ is comeager. 
The proof that $P(\ell)$ is comeager is analogous. 
\end{proof}

We are finally ready to prove Theorem \ref{thm:theoremcluster}.
%The proof of Theorem \ref{thm:theoremcluster} proceeds along the same lines of \cite[Theorem 2.1]{MR3968131}. 
%We include it here for the sake of completeness. 
%The proof of Theorem \ref{thm:theoremclusterpermutations} goes verbatim, hence it is omitted.
\begin{proof}
[Proof of Theorem \ref{thm:theoremcluster}] 
%The proof proceeds verbatim as in \cite[Theorem 2.1]{MR3968131}, replacing \cite[Lemma 3.2]{MR3968131} with Lemma \ref{lem:fsigmanew}.
\ref{item:s1} $\implies$ \ref{item:s2} It is obvious.

\ref{item:s2} $\implies$ \ref{item:s3} Suppose that there exists $\ell \in \clusterfin \setminus \Gamma_x(\mathcal{I})$. Then $\Sigma_x(\mathcal{I})$ is contained in $\Sigma\setminus S(\ell)$, which is meager by Lemma \ref{lem:fsigmanew}. 

\ref{item:s3} $\implies$ \ref{item:s1} Suppose that $\clusterfin \neq \emptyset$, otherwise the claim is trivial. Let $\mathscr{L}$ be a countable dense subset of $\clusterfin$, so that $\mathscr{L}\subseteq \Gamma_{\sigma(x)}(\mathcal{I})$ for each $\sigma \in S:=\bigcap_{\ell \in \mathscr{L}}S(\ell)$, which is comeager by Lemma \ref{lem:fsigmanew}. Fix $\sigma \in S$. On the one hand, $\Gamma_{\sigma(x)}(\mathcal{I}) \subseteq \mathrm{L}_{\sigma(x)}\subseteq \clusterfin$. On the other hand, since $\Gamma_{\sigma(x)}(\mathcal{I})$ is closed by \cite[Lemma 3.1(iv)]{MR3920799}, we get $\clusterfin \subseteq \Gamma_{\sigma(x)}(\mathcal{I})$.  Therefore $S\subseteq \Sigma_x(\mathcal{I})$. 
%$\clusterfin=\Gamma_{x _\sigma}(\mathcal{I})$ for all $\sigma \in S$.

The implications \ref{item:p1} $\implies$ \ref{item:p2} $\implies$ \ref{item:s3} $\implies$ \ref{item:p1} are analogous.
\end{proof}

%\begin{proof}
%[Proof of Theorem \ref{thm:theoremclusterpermutations}] 
%It is similar to the proof of Theorem \ref{thm:theoremcluster}, replacing Lemma \ref{lem:fsigmanew} with Lemma \ref{lem:fsigmanewperm} and noting that $\mathrm{L}_{\pi(x)}=\clusterfin$ for all $\pi \in \Pi$.
%\end{proof}

\begin{rmk}\label{rmk:uncountable}
%In the proofs of Theorems \ref{thm:theoremcluster} and \ref{thm:theoremclusterpermutations}, 
As it is evident from the proof above, 
the hypothesis that \textquotedblleft closed sets of $X$ are separable\textquotedblright\,can be removed if, in addition, $\clusterfin$ is countable.
\end{rmk}

\begin{lem}\label{lem:clusterpoint}
Let $\mathcal{I}$ be an ideal and $x$ be a sequence in a first countable compact space. 
%space $X$ such that $\{n \in \mathbf{N}: x_n \notin K\} \in \mathcal{I}$ for some compact $K\subseteq X$. 
Then $x \to_{\mathcal{I}} \ell$ if and only if $\Gamma_x(\mathcal{I})=\{\ell\}$.
\end{lem}
\begin{proof}
It follows by \cite[Corollary 3.4]{MR3920799}.
\end{proof}

\begin{proof}
[Proof of Corollary \ref{thm:applicationidealconvergence}] 
\ref{item:c2} $\implies$ \ref{item:c3} It is obvious.

\ref{item:c3} $\implies$ \ref{item:c1} By Lemma \ref{lem:clusterpoint}, the hypothesis can be rewritten as $\Gamma_x(\mathcal{I})=\{\ell\}$. Hence, condition \ref{item:c3} is equivalent to the nonmeagerness of $\{\sigma \in \Sigma: \Gamma_{\sigma(x)}=\Gamma_x(\mathcal{I})=\{\ell\}\}$. The claim follows by Theorem \ref{thm:theoremcluster} and Remark \ref{rmk:uncountable}. 

\ref{item:c1} $\implies$ \ref{item:c2} If $x\to\ell$ then $\sigma(x)\to_{\mathcal{I}} \ell$ for all $\sigma \in \Sigma$. 

The implications \ref{item:c4} $\implies$ \ref{item:c5} $\implies$ \ref{item:c1} $\implies$ \ref{item:c4} are analogous.
\end{proof}

\section{Proofs for $\mathcal{I}$-limit points}\label{sec:prooflimit}

A lower semicontinuous submeasure (in short, lscsm) is a monotone subadditive function $\varphi: \mathcal{P}(\mathbf{N}) \to [0,\infty]$ such that $\varphi(\emptyset)=0$, $\varphi(F)<\infty$ for all $F \in \mathrm{Fin}$, and $\varphi(A)=\lim_n \varphi(A\cap [1,n])$ for all $A\subseteq \mathbf{N}$. By a classical result of Solecki, an ideal $\mathcal{I}$ is an analytic P-ideal if and only if there exists a lscsm $\varphi$ such that
\begin{equation}\label{eq:characterizationanalyticPideal}
\mathcal{I}=\mathrm{Exh}(\varphi):=\{A\subseteq \mathbf{N}: \|A\|_\varphi=0\} 
\,\,\,\text{ and }\,\,\,
0<\|\mathbf{N}\|_\varphi \le \varphi(\mathbf{N})<\infty,
\end{equation}
where $\|A\|_\varphi:=\lim_n \varphi(A\setminus [1,n])$ for all $A\subseteq \mathbf{N}$, see \cite[Theorem 3.1]{MR1708146}. Note that $\|\cdot\|_\varphi$ is a submeasure which is invariant modulo finite sets. Moreover, replacing $\varphi$ with $\varphi/ \|\mathbf{N}\|_\varphi$  in \eqref{eq:characterizationanalyticPideal}, we can assume without loss of generality that $\|\mathbf{N}\|_\varphi=1$. 

Given a sequence $x$ in a first countable topological space $X$ and an analytic P-ideal $\mathcal{I}=\mathrm{Exh}(\varphi)$, we define the function
\begin{equation}\label{eq:definitionufrak}
\mathfrak{u}: \Sigma \times X \to \mathbf{R}: (\sigma,\ell) \mapsto \lim_{k\to \infty} \|\{n \in \mathbf{N}: x_{\sigma(n)} \in U_k\}\|_\varphi,
\end{equation}
where $(U_k)$ is a decreasing local base of neighborhoods at $\ell \in X$. Clearly, the limit in \eqref{eq:definitionufrak} exists and it is independent of the choice of $(U_k)$. 
%
%It follows by \cite[Lemma 2.1]{MR3883171} that the section $\mathfrak{u}(\sigma,\cdot)$ is upper semicontinuous for each $\sigma \in \Sigma$. 

\begin{lem}\label{lem:uppersemicontinuous}
Let $x$ be a sequence in a first countable space $X$ and let $\mathcal{I}=\mathrm{Exh}(\varphi)$ be an analytic P-ideal. 
Then, the section $\mathfrak{u}(\sigma,\cdot)$ is upper semicontinuous for each $\sigma \in \Sigma$. In particular, the set
$$
%\Lambda_x(\mathcal{I},q):=\{\ell \in X: \lim_{k\to \infty}\|\{n: x_n \in U_k\}\|_\varphi\ge q\}
\Lambda_{\sigma(x)}(\mathcal{I},q):=\{\ell \in X: \mathfrak{u}(\sigma,\ell)\ge q\}
$$
is closed for all $q>0$.
\end{lem}
\begin{proof}
See \cite[Lemma 2.1]{MR3883171}.
\end{proof}

\begin{lem}\label{lem:analoguecomeagerold}
%Let $x$ be a sequence in a first countable space $X$ and let $\mathcal{I}=\mathrm{Exh}(\varphi)$ be an analytic P-ideal. 
With the same hypotheses of Lemma \ref{lem:uppersemicontinuous}, 
the set 
$$
V(\ell,q):=\{\sigma \in \Sigma: \mathfrak{u}(\sigma,\ell)> q\}
$$
is either comeager or empty for each $\ell \in X$ and $q \in (0,1)$.
\end{lem}
\begin{proof}
Suppose that $V(\ell,q) \neq \emptyset$, so that $\ell \in \clusterfin$, 
%so that it contains $\sigma_0 \in \Sigma$, 
and note that 
%$V(\ell,q)^c=\bigcup_k W_k$, where 
%$$
%\forall k \in \mathbf{N},\quad W_k:=\{\sigma \in \Sigma: \|\{n \in \mathbf{N}: x_{\sigma(n)} \in U_k\}\|_\varphi \le q\}.
%$$
\begin{displaymath}
\begin{split}
\Sigma\setminus V(\ell,q)&=\bigcup_{k\ge 1}\{\sigma \in \Sigma: \|\{n \in \mathbf{N}: x_{\sigma(n)} \in U_k\}\|_\varphi \le q\}\\
&=\bigcup_{k\ge 1}\{\sigma \in \Sigma: \limsup_{t\to \infty}\varphi(\{n \ge t: x_{\sigma(n)} \in U_k\}) \le q\}\\
&=\bigcup_{k\ge 1}\bigcup_{s\ge 1} \bigcap_{t\ge s}\{\sigma \in \Sigma: \varphi(\{n \ge t: x_{\sigma(n)} \in U_k\}) \le q\}.
\end{split}
\end{displaymath}
Then, it is sufficient to show that 
$$
W_{k,s}:=\bigcap_{t\ge s}\{\sigma \in \Sigma: \varphi(\{n \ge t: x_{\sigma(n)} \in U_k\}) \le q\}
$$
is nowhere dense for all $k,s \in \mathbf{N}$. 

To this aim, for every nonempty open set $Z\subseteq \Sigma$, we need to prove that there exists a nonempty open subset $S \subseteq Z$ such that $S \cap W_{k,s}=\emptyset$. Fix a nonempty open set $Z\subseteq \Sigma$ and $\sigma_0 \in Z$ so that there exists $n_0 \in \mathbf{N}$ for which
$$
Z^\prime:=\{\sigma \in \Sigma: \sigma \upharpoonright \{1,\ldots,n_0\}=\sigma_0 \upharpoonright \{1,\ldots,n_0\}\} \subseteq Z.
$$
Since $\ell \in \clusterfin$, there exists $\sigma_1 \in Z^\prime$ such that $\lim_n x_{\sigma_1(n)}=\ell$. Therefore %$\mathfrak{u}(\sigma_1,\ell)=1$ and, in particular, 
\begin{displaymath}
\begin{split}
\varphi(\{n \ge n_1: x_{\sigma_1(n)} \in U_k\})&\ge \|\{n \ge n_1: x_{\sigma_1(n)} \in U_k\}\|_\varphi \\
&=\|\{n \in \mathbf{N}: x_{\sigma_1(n)} \in U_k\}\|_\varphi=\mathfrak{u}(\sigma_1,\ell)=1,
\end{split}
\end{displaymath}
where $n_1:=\max\{n_0+1,s\}$. 
At this point, since $\varphi$ is a lscsm, it follows that there exists an integer $n_2>n_1$ such that $\varphi(\{n \in \mathbf{N} \cap [n_1, n_2]: x_{\sigma_1(n)} \in U_k\})> q$. Therefore $
S:=\{\sigma \in Z^\prime: \sigma \upharpoonright \{n_1,\ldots,n_2\}=\sigma_1 \upharpoonright \{n_1,\ldots,n_2\}\}
$ 
is a nonempty open set contained in $Z$ and disjoint from $W_{k,s}$. Indeed 
$$
\forall \sigma \in S, \quad \varphi(\{n \ge s: x_{\sigma(n)} \in U_k\}) \ge \varphi(\{n \in \mathbf{N} \cap [n_1, n_2]: x_{\sigma(n)} \in U_k\})> q
$$
by the monotonicity of $\varphi$. 
\end{proof}

\begin{lem}\label{lem:analoguecomeageroldsecond}
With the same hypotheses of Lemma \ref{lem:uppersemicontinuous}, we have
%For each $\ell \in X$, we have 
$$
\textstyle 
\forall \ell \in X,\quad 
%\tilde{S}(\ell)
%:=
\{\sigma \in \Sigma: \ell \in \Lambda_{\sigma(x)}(\mathcal{I})\}
=\bigcup_{q>0}V(\ell,q).
$$
%where $\tilde{S}(\ell):=\{\sigma \in \Sigma: \ell \in \Lambda_{\sigma(x)}(\mathcal{I})\}$.
In addition, $\tilde{S}(\ell,q):=\{\sigma \in \Sigma: \ell \in \Lambda_{\sigma(x)}(\mathcal{I},q)\}$ contains $V(\ell,q)$.
\end{lem}
\begin{proof}
Fix $\ell \in X$ and $\sigma \in \tilde{S}(\ell)$, where
$$
\tilde{S}(\ell)
:=
\{\sigma \in \Sigma: \ell \in \Lambda_{\sigma(x)}(\mathcal{I})\}.
$$
Then there exist $\tau \in \Sigma$ and $q>0$ such that $\lim_n x_{\tau(\sigma(n))}=\ell$ and $\|\tau(\mathbf{N})\|_\varphi \ge 2q$.  In particular, for each $k \in \mathbf{N}$ we have $x_{\tau(\sigma(n))} \in U_k$ for all large $n \in \mathbf{N}$. Hence 
\begin{displaymath}
\|\{n \in \mathbf{N}: x_{\sigma(n)} \in U_k\}\|_\varphi \ge \|\{n \in \mathbf{N}: x_{\sigma(n)} \in U_k\} \cap \tau(\mathbf{N})\|_\varphi =\| \tau(\mathbf{N})\|_\varphi \ge 2q.
\end{displaymath}
By the arbitrariness of $k$, it follows that $\mathfrak{u}(\sigma,\ell)\ge 2q>q$, that is, $\sigma \in V(\ell,q)$. 

Conversely, fix $\ell \in X$, $\sigma \in \Sigma$, and $q>0$ such that $\sigma \in V(\ell,q)$, hence $\|A_k\|_\varphi>q$ for all $k \in \mathbf{N}$, where $A_k:=\{n \in \mathbf{N}: x_{\sigma(n)} \in U_k\}$. 
Let us define recursively a sequence $(F_k)$ of finite subsets of $\mathbf{N}$ as it follows. Pick $F_1\subseteq A_1$ such that $\varphi(F_1) \ge q$ (which is possibile since $\varphi$ is a lscsm); then, for each integer $k\ge 2$, let $F_k$ be a finite subset of $A_k$ such that $\min F_k>\max F_{k-1}$ and $\varphi(F_k) \ge q$ (which is possible since $\|A_k \setminus [1,\max F_{k-1}]\|_\varphi=\|A_k\|_\varphi>q$). Let $(y_n)$ be the increasing enumeration of the set $\bigcup_k F_k$, and define $\tau \in \Sigma$ such that $\tau(n)=y_n$ for all $n$. It follows by construction that %$\lim_n x_{\tau(\sigma(n))}=\ell$ and 
$$
\lim_{n\to \infty} x_{\tau(\sigma(n))}=\ell
\,\,\,\text{ and }\,\,\,
\|\tau(\mathbf{N})\|_\varphi \ge \liminf_{k\to \infty} \varphi(F_k) \ge q>0.
$$
Therefore $\ell \in \Lambda_{\sigma(x)}(\mathcal{I},q)\subseteq \Lambda_{\sigma(x)}(\mathcal{I})$, which concludes the proof.
\end{proof}

\begin{cor}\label{lem:fsigmanewlimit}
%Let $x$ be a sequence in a first countable space $X$ and let $\mathcal{I}$ be an analytic P-ideal. Then 
With the same hypotheses of Lemma \ref{lem:uppersemicontinuous}, 
$\tilde{S}(\ell,q)$ 
%$$
%\tilde{S}(\ell):=\left\{\sigma \in \Sigma: \ell \in \Lambda_{\sigma(x)}(\mathcal{I})\right\}
%$$ 
is comeager for each $\ell \in \clusterfin$ and $q \in (0,1)$.
\end{cor}
\begin{proof}
Fix $\ell \in \clusterfin$ and $q \in (0,1)$. Then $\tilde{S}(\ell,q)$ contains $V(\ell,q)$ by Lemma \ref{lem:analoguecomeageroldsecond}, which is comeager by Lemma \ref{lem:analoguecomeagerold}.
%
%
%Then there exists $\sigma \in \Sigma$ such that $\lim_n x_{\sigma(n)}=\ell$. Hence $\sigma \in V(\ell, q)$, which is comeager by Lemma \ref{lem:analoguecomeagerold}. Since $V(\ell,q) \subseteq \tilde{S}(\ell)$ by Lemma \ref{lem:analoguecomeageroldsecond}, the claim follows.
\end{proof}

\begin{cor}\label{lem:fsigmanewlimit2}
With the same hypotheses of Lemma \ref{lem:uppersemicontinuous}, $\tilde{S}(\ell)$ is comeager for each $\ell \in \clusterfin$.
\end{cor}
\begin{proof}
Thanks to \cite[Theorem 2.2]{MR3883171}, we have
\begin{equation}\label{eq:inclusionlimitpoints}
\textstyle \Lambda_{\sigma(x)}(\mathcal{I})=\bigcup_{q>0}\Lambda_{\sigma(x)}(\mathcal{I},q).
\end{equation}
Therefore $\tilde{S}(\ell)$ contains $\tilde{S}(\ell,\nicefrac{1}{2})$, which is comeager by Corollary \ref{lem:fsigmanewlimit}.
%
%\bigskip
%
%Then there exists $\sigma \in \Sigma$ such that $\lim_n x_{\sigma(n)}=\ell$.
%Hence $\sigma \in V(\ell, q)$, which is comeager by Lemma \ref{lem:analoguecomeagerold}. Since $V(\ell,q) \subseteq \tilde{S}(\ell)$ by Lemma \ref{lem:analoguecomeageroldsecond}, the claim follows.
\end{proof}

\begin{rmk}\label{rmk:lastpermutations}
All the analogues from Lemma \ref{lem:uppersemicontinuous} up to Corollary \ref{lem:fsigmanewlimit2} hold for permutations, the only difference being in the last part of the proof of Lemma \ref{lem:analoguecomeagerold}: let us show that
$$
\widehat{W}_{k,s}:=\bigcup_{t\ge s}\{\pi \in \Pi: \varphi(\{n\ge t: x_{\pi(n)} \in U_k\})\le q\}
$$
is nowhere dense for all $k,s \in \mathbf{N}$. To this aim, fix $\pi_0 \in \Pi$ and $n_0 \in \mathbf{N}$ which defines the nonempty open set 
$
G:=\{\pi \in \Pi: \pi \upharpoonright \{1,\ldots,n_0\}=\pi_0 \upharpoonright \{1,\ldots,n_0\}\}.
$
Set $n_1:=\max\{n_0+1,s\}$ and let $(y_n)$ be the increasing enumeration of the infinite set $\{n \in \mathbf{N}: x_n \in U_k\} \setminus \{\pi_0(1),\ldots,\pi_0(n_0)\}$. 
Since $\varphi$ is a lscsm, there exists $n_2 \in \mathbf{N}$ such that $\varphi(\{s,s+1,\ldots,n_2\})>q$. Lastly, let $G^\prime$ be the set of all $\pi \in G$ such that $\pi(n)=y_n$ for all $n \in \{s,s+1,\ldots,n_2\}$. We conclude that
$$
\forall \pi \in G^\prime, \quad \varphi(\{n\ge s: x_{\pi(n)} \in U_k\}) \ge \varphi(\{s,s+1,\ldots,n_2\})>q.
$$
Therefore $G^\prime$ is a nonempty open subset of $G$ which is disjoint from $\widehat{W}_{k,s}$.
\end{rmk}

%\begin{cor}\label{cor:ilimitpointimportant}
%With the same hypotheses of Lemma \ref{lem:analoguecomeagerold}, the set 
%$$
%\tilde{S}(\ell,q):=\{\sigma \in \Sigma: \ell \in \Lambda_{\sigma(x)}(\mathcal{I},q)\}
%$$
%is comeager for each $\ell \in \clusterfin$ and $q \in (0,1)$.
%\end{cor}
%\begin{proof}
%Fix $\ell \in \clusterfin$ and $q \in (0,1)$. 
%%Then there exists $\sigma \in \Sigma$ such that $\lim_n x_{\sigma(n)}=\ell$. 
%%%In particular, $\ell \in \Lambda_{\sigma(x)}(\mathcal{I})$. 
%%%Thanks to \cite[Theorem 2.2]{MR3883171}, we have 
%%%$$
%%%\textstyle \Lambda_{\sigma(x)}(\mathcal{I})=\bigcup_{t>0}\Lambda_{\sigma(x)}(\mathcal{I},t).
%%%$$
%%Hence $\ell \in \Lambda_{\sigma(x)}(\mathcal{I},q)$ and $V(\ell,\nicefrac{q}{2})\neq \emptyset$. The claims follows by the fact that $\tilde{S}(\ell,q)$ contains $V(\ell,\nicefrac{q}{2})$, which is comeager by Lemma \ref{lem:analoguecomeagerold}.
%Thanks to \cite[Theorem 2.2]{MR3883171}, we have 
%$$
%\textstyle \Lambda_{\sigma(x)}(\mathcal{I})=\bigcup_{t>0}\Lambda_{\sigma(x)}(\mathcal{I},t).
%$$
%Therefore $\Lambda_{\sigma(x)}(\mathcal{I},q) \subseteq \Lambda_{\sigma(x)}(\mathcal{I})$.
%\end{proof}

Lastly, we prove Theorem \ref{thm:theoremlimit}.
\begin{proof}
[Proof of Theorem \ref{thm:theoremlimit}] 
The implications \ref{item:l1} $\implies$ \ref{item:l2} $\implies$ \ref{item:l3} are analogous to the ones in Theorem \ref{thm:theoremcluster}, replacing Lemma \ref{lem:fsigmanew} with Corollary \ref{lem:fsigmanewlimit2}.

%\ref{item:l2} $\implies$ \ref{item:l3} Suppose that there exists $\ell \in \clusterfin \setminus \Lambda_x(\mathcal{I})$. Then $\tilde{\Sigma}_x(\mathcal{I})$ is contained in $\Sigma\setminus \tilde{S}(\ell)$, which is meager by Lemma \ref{lem:fsigmanewlimit}. 

\ref{item:l3} $\implies$ \ref{item:l1} Suppose that $\clusterfin \neq \emptyset$, otherwise the claim is trivial. Let $\mathscr{L}$ be a countable dense subset of $\clusterfin$, so that $\mathscr{L}\subseteq \Lambda_{\sigma(x)}(\mathcal{I}, \nicefrac{1}{2})$ for each $\sigma \in \tilde{S}:=\bigcap_{\ell \in \mathscr{L}}\tilde{S}(\ell, \nicefrac{1}{2})$, which is comeager by Corollary \ref{lem:fsigmanewlimit}. Fix $\sigma \in \tilde{S}$. 
%and recall, thanks to \cite[Theorem 2.2]{MR3883171}, that
%$$
%\textstyle \Lambda_{\sigma(x)}(\mathcal{I})=\bigcup_{q>0}\Lambda_{\sigma(x)}(\mathcal{I},q).
%$$
On the one hand, taking into account \eqref{eq:inclusionlimitpoints}, we get $\Lambda_{\sigma(x)}(\mathcal{I}, \nicefrac{1}{2})\subseteq \Lambda_{\sigma(x)}(\mathcal{I}) \subseteq \mathrm{L}_{\sigma(x)}\subseteq \clusterfin$. On the other hand, since $\Lambda_{\sigma(x)}(\mathcal{I},\nicefrac{1}{2})$ is closed by Lemma \ref{lem:uppersemicontinuous}, we obtain $\clusterfin \subseteq \Lambda_{\sigma(x)}(\mathcal{I})$. Therefore $\tilde{\Sigma}_x(\mathcal{I})$ contains the comeager set $\tilde{S}$. 

The implications \ref{item:llll1} $\implies$ \ref{item:llll2} $\implies$ \ref{item:l3} $\implies$ \ref{item:llll1} are analogous, taking into account Remark \ref{rmk:lastpermutations}.
\end{proof}

\subsection*{Acknowledgments}
The authors are grateful to the anonymous referee for several accurate and constructive comments which led, in particular, to a simplification of the proof of Lemma \ref{lem:fsigmanew}.

%\begin{thebibliography}{99}

%%%%%%%%%%%%%%%%%%%%%%%%%%%%%%%%%%%%%%%%%%%%%%%%%%%%%%%%%%%%%%%%%%%%%%%
%\nocite{*}
\bibliographystyle{amsplain}
%\bibliography{subsequences}

\end{document}